\documentclass[10pt]{article}

\usepackage[margin=1in]{geometry}  
\usepackage{graphicx}              
\usepackage{amsmath}               
\usepackage{amsfonts}              
\usepackage{amsthm}                
\usepackage{amssymb}

\usepackage{comment}

\makeatletter
\newcommand*{\defeq}{\mathrel{\rlap{%
                     \raisebox{0.3ex}{$\m@th\cdot$}}%
                     \raisebox{-0.3ex}{$\m@th\cdot$}}%
                     =}
\makeatother

\newtheorem{theorem}{Theorem}[section]

\newtheorem{corollary}[theorem]{Corollary}

\newtheorem{example}[theorem]{Example}

\begin{document}

\nocite{*}

\title{A weak criterion of bigness for toric vector bundles}

\author{Evgeny Mayanskiy}

\maketitle

\begin{abstract}
  We prove a weak version of a bigness criterion for equivariant vector bundles on toric varieties. 
\end{abstract}

\section{Preliminaries}

Let $k$ be an algebraically closed field of characteristic $0$, $T=(k^{*})^{r}$ a torus, $M=M(T)\cong {\mathbb Z}^{r}$ and $N={M(T)}^{\vee}$ the character and cocharacter lattices with the duality pairing $(m,n)\defeq n(m)\in \mathbb Z$, $m\in M$, $n\in N$.\\

We use notation from \cite{Fulton}. Let $\Delta$ be a fan in $N$ such that the associated toric variety $X=X(\Delta)$ is projective. Let $P$ be the set of rays of $\Delta$ and $v_{\rho}\in N$, $\rho \in P$, their generators, i.e. $\lvert{\rho}\rvert \cap N = {\mathbb Z}_{\geq 0} \cdot v_{\rho}$. For any cone $\sigma \in \Delta$, let $T_{\sigma}\subset T$ denote the stabilizer of a point of the $T$-orbit in $X$ corresponding to $\sigma$. Note that $M(T_{\sigma}) = M/({\sigma}^{\perp}\cap M)$. We assume that all vector spaces and bundles have finite rank.\\

According to Klyachko \cite{Klyachko}, a $T$-equivariant vector bundle $\mathcal E$ on $X$, customarily called 'a toric vector bundle', is determined, up to an isomorphism, by a collection of complete vector space $\mathbb Z$-filtrations
$$
{\mathcal F}_{\rho} (\mathcal E) = \{ E \supseteq \cdots \supseteq {\mathcal E}^{\rho}(j) \supseteq {\mathcal E}^{\rho}(j+1) \supseteq \cdots \supseteq (0)  \}, \; \rho\in P, 
$$
of the fiber $E={\mathcal E}_{o}$ over $o=(1,\ldots ,1)\in T\subset X$ satisfying the following compatibility condition:
\begin{center}
\begin{itshape}
for any cone $\sigma \in \Delta$, there is a grading $E = \bigoplus_{\substack{\chi \in M(T_{\sigma})}} E_{\chi}^{\sigma}$ such that for any $\rho\in P$, $\rho\prec \sigma$, $j\in \mathbb Z$
\begin{equation*}
{\mathcal E}^{\rho}(j) = \bigoplus_{\substack{\chi \in M(T_{\sigma}) \\ (\chi , v_{\rho})\geq j}} E_{\chi}^{\sigma}.
\end{equation*}
\end{itshape}
\end{center}

One can similarly define filtrations of $E$ for any $n \in N$:
$$
{\mathcal E}^{n}(j) = \bigoplus_{\substack{\chi \in M(T_{\sigma}) \\ (\chi , n)\geq j}} E_{\chi}^{\sigma}\;\; \text{if}\;\; n\in \lvert\sigma\rvert .
$$
Given a toric vector bundle $\mathcal E$ on $X$ and $e\in E \setminus \{ 0 \}$, one can associate with them
\begin{itemize}
\item a piecewise linear function ${\phi}_e\colon N\otimes_{\mathbb Z}{\mathbb Q} \rightarrow {\mathbb Q}$ defined by
$$
{\phi}_e (n) = \max\limits_{j} \{ j \mid e\in {\mathcal E}^{n}(j) \}, \; n\in N,\;\text{and}
$$
\item a convex polytope ${\Delta}_e\subset M\otimes_{\mathbb Z}{\mathbb Q}$ defined by the function ${\phi}_e$:
$$
{\Delta}_e = \{ u \in M\otimes_{\mathbb Z}{\mathbb Q} \mid (u,v_{\rho}) \leq {\phi}_e (v_{\rho}),\; \rho\in P \} .
$$
\end{itemize}

Note that 
$$
H^0(X,\mathcal E)=\bigoplus_{u\in M} \operatorname*{Span} \{ {{\chi}^{-u}}\otimes e \mid u\in {\Delta}_e \} \subset k[M]\otimes_{k}{E}\quad \cite{Klyachko}.
$$

Moreover, there is a \textit{finite} subset $\epsilon (X,\mathcal E)\subset E \setminus \{ 0 \}$ such that 
$$
H^0(X,\mathcal E)=\operatorname*{Span}_{e\in \epsilon (X,\mathcal E)} \{ {{\chi}^{-u}}\otimes e \mid u\in {\Delta}_e\cap M \}.
$$
One may take $\epsilon (X,\mathcal E)$ to be the ground set of the matroid constructed in \cite{DiRocco}. Let $\overline{\epsilon}(X,\mathcal E)=\{ e\in {\epsilon}(X,\mathcal E) \mid {\Delta}_e\cap M \neq \varnothing \} .$\\

The vector bundle $\mathcal E$ is big if and only if there is an integer $p > 0$ such that
$$
\limsup\limits_{l \rightarrow \infty} \frac{\dim \operatorname{Im} (S^{l}H^{0}(X,\operatorname{Sym}^{p}{\mathcal E}) \rightarrow H^{0}(X,\operatorname{Sym}^{pl}{\mathcal E}))}{(pl)^d/d!} > 0,
$$
i.e.
\begin{equation}\label{S}
\limsup\limits_{l \rightarrow \infty} \left( \frac{1}{l^{d}} \cdot \dim \operatorname*{Span}_{e_1,\ldots , e_l\in \epsilon (X,\operatorname{Sym}^{p}{\mathcal E})} \{ {\chi}^{-(u_1 + \ldots + u_l)}\otimes (e_1 \cdots e_l) \mid u_i\in {\Delta}_{e_i}\cap M,\; i=1,\ldots , l \} \right) > 0.
\end{equation}
Here $d=\dim X + \operatorname{rk} {\mathcal E} - 1$. A convenient reference is \cite{Lazarsfeld}, Theorem~$3.3$, or work of Khovanskii.\\

If $p > 0$, then the vector bundle $\operatorname{Sym}^{p}{\mathcal E}$, the fiber of which over $o$ is $\operatorname{Sym}^{p}{E}$,  corresponds to the following filtrations:
$$
{\left( \operatorname{Sym}^{p}{\mathcal E} \right)}^{\rho}(j) = \sum_{j_1 + \ldots + j_p = j} {\mathcal E}^{\rho}(j_1)\cdot \ldots \cdot {\mathcal E}^{\rho}(j_p) \subset \operatorname{Sym}^{p}{E}, \; j\in\mathbb Z ,\; \rho\in P.
$$

This implies that if $f={f}^{c_1}_1\ldots {f}^{c_q}_q\in \operatorname{Sym}^{\sum_{i} c_ip_i}{E}$, $f_i\in \operatorname{Sym}^{p_i}{E}$, $c_i\geq 0$, $i=1,\ldots ,q$, then
$$
c_1\cdot {\Delta}_{f_1} + \ldots + c_q\cdot {\Delta}_{f_q} \subset {\Delta}_f, 
$$
where $+$ denotes Minkowski addition.

\section{The main result}

Given a polytope $\pi\subset {\mathbb Q}^r$, let $L\pi = \operatorname*{Span} \{ a - b \mid a\in \pi, b\in \pi \} \subset {\mathbb Q}^r$ be the vector subspace spanned by the differences of points of $\pi$. Similarly, the toric vector bundle $\mathcal E$ on $X$ defines a vector subspace
$$
L(X,\mathcal E) = \operatorname*{Span} \{ a - b \mid a, b\in {\Delta}_f ,\; f\in \operatorname{Sym}^{p}{E} \setminus \{ 0 \} ,\; p > 0 \} \subset M\otimes_{\mathbb Z}{\mathbb Q}.
$$
If $w\colon M\otimes_{\mathbb Z}{\mathbb Q} \rightarrow M\otimes_{\mathbb Z}{\mathbb Q} / L(X,\mathcal E)$ is the natural linear projection, then we define a subsemigroup
$$
\Lambda = \{ w({\Delta}_f) \mid f\in \operatorname{Sym}^{p}{E} \setminus \{ 0 \} ,\; p > 0 ,\; {\Delta}_f \neq \varnothing \}  \subset M\otimes_{\mathbb Z}{\mathbb Q} / L(X,\mathcal E).
$$

Note that, unless $\mathcal E = 0$, $L(X,\mathcal E) = L{\Delta}_e$ for some $ e\in \operatorname{Sym}^{a}{E} \setminus \{ 0 \} $, $a > 0$. Moreover, one can then find such $e$ in the ground set of the matroid constructed in \cite{DiRocco} for $\operatorname{Sym}^{a}{\mathcal E}$.\\

For any $p > 0$, consider the finite set $W_p \subset M\otimes_{\mathbb Z}{\mathbb Q} / L(X,\mathcal E)$ of points
$$
w_f = w({\Delta}_{f}), \; f\in \overline{\epsilon} (X,\operatorname{Sym}^{p}{\mathcal E}) .
$$
Let ${\Lambda}_p \subset \Lambda\times\mathbb Z$ be the subsemigroup generated by $W_p\times \{ 1 \}$ and ${\Delta}^{\perp}(p,X,\mathcal E) \subset M\otimes_{\mathbb Z}{\mathbb Q} / L(X,\mathcal E)$ be the convex hull of $W_p$. Consider the following algebra graded by ${\Lambda}_p$ and generated by $\overline{\epsilon} (X,\operatorname{Sym}^{p}{\mathcal E})=\{ f_1,\ldots ,f_q \}$:
$$
{\mathcal A}(p,X,\mathcal E)=\bigoplus_{\overline{w}\in{\Lambda}_p} {\mathcal A}_{\overline{w}}(p,X,\mathcal E),
$$
$$
{\mathcal A}_{(w,l)}(p,X,\mathcal E)=\operatorname*{Span} \{ {f}^{c_1}_1\ldots {f}^{c_q}_q \mid c_i\geq 0,\; c_1\cdot w_{f_1} + \ldots + c_q\cdot w_{f_q} = w,\; c_1 + \ldots + c_q= l \} \subset \operatorname{Sym}^{pl}{E}.
$$
Multiplication in ${\mathcal A}(p,X,\mathcal E)$ is inherited from the symmetric algebra $\operatorname{Sym}{E}$. Let
$$
{\alpha}(p,X,\mathcal E)=\limsup\limits_{l\rightarrow \infty} \sum_{w\in \Lambda , \; (w,l)\in {\Lambda}_p}\dim {\mathcal A}_{(w,l)}(p,X,\mathcal E) / l^{\dim X - \dim L(X,\mathcal E)+\operatorname{rk}{\mathcal E} - 1}\in\mathbb R.
$$

Assume that the subset
$$
\bigcup_{p>0} {\epsilon}(X,\operatorname{Sym}^{p}{\mathcal E}) \subset \bigoplus_{p\geq 0}\operatorname{Sym}^{p}{E} = \operatorname{Sym}{E}
$$
of the symmetric algebra is closed under multiplication and contains the ground sets of the matroids constructed in \cite{DiRocco} for $\operatorname{Sym}^{p}{\mathcal E}$, $p>0$.\\

The main result of this note is the following criterion of bigness.

\begin{theorem}\label{theorem:main}
Let $X=X(\Delta)$ be a projective toric variety, $\mathcal E$ a toric vector bundle on $X$. Then the following are equivalent:
\begin{enumerate}
\item $\mathcal E$ is big,
\item there is an integer $p > 0$ such that ${\alpha}(p,X,\mathcal E) > 0$.
\end{enumerate}
\end{theorem}

\begin{proof}

We may assume that $\dim X > 0$.\\

If $\mathcal E$ is big, then there is an integer $p>0$ such that \eqref{S} holds. If $\overline{\epsilon}(X,\operatorname{Sym}^{p}{\mathcal E}) = \{ f_1,\ldots ,f_q \}$, then \eqref{S} implies that
$$
0 < \limsup\limits_{l\rightarrow \infty} \left( \frac{1}{l^{d}} \cdot \dim \operatorname*{Span}_{\substack{c_1+\ldots +c_q = l \\ c_i \geq 0}} \{ {\chi}^{-u}\otimes {f^{c_1}_{1}\cdots f^{c_q}_{q}} \mid u \in \left( c_1 \cdot {\Delta}_{f_1} + \ldots + c_q \cdot {\Delta}_{f_q}\right)\cap M \} \right).
$$

This expression is bounded from above by
$$
\left( \limsup\limits_{l\rightarrow \infty} \frac{\#\left( M \cap l\cdot \left( {{\Delta}_{f_1}}+ \ldots + {{\Delta}_{f_q}} \right) \right)}{l^{\dim L(X,\mathcal E)}} \right) \cdot {\alpha}(p,X,\mathcal E),
$$
and so ${\alpha}(p,X,\mathcal E) > 0$.\\

Suppose that there is an integer $p>0$ such that ${\alpha}(p,X,\mathcal E) > 0$. We may assume that $L(X,\mathcal E) = L{\Delta}_e$ for some $e\in \overline{\epsilon}(X,\operatorname{Sym}^{p}{\mathcal E}) = \{ f_1,\ldots , f_q \}$. Choose an integer $N > 0$ large enough so that
$$
\bigcap_{\substack{c_1,\ldots , c_q\in\mathbb R \\ \max_i \lvert c_i \rvert \leq 1/N}} (1+c_1)\cdot \tilde{{\Delta}}_{f_1} + \ldots + (1+c_q)\cdot \tilde{{\Delta}}_{f_q}
$$
has a nonempty interior in $L(X,\mathcal E)$, where $\tilde{{\Delta}}_{f_i}$ denotes the orthogonal projection of ${\Delta}_{f_i}$ onto $L(X,\mathcal E)$ with respect to some positive definite integral symmetric bilinear form on $M$. In particular, every polytope $(1+c_1)\cdot \tilde{{\Delta}}_{f_1} + \ldots + (1+c_q)\cdot \tilde{{\Delta}}_{f_q}$, where $\max_i \lvert c_i \rvert \leq 1/N$, contains a fixed ball of dimension $\dim L(X,\mathcal E)$ and radius $s > 0$. Thus, if $f=(f_1\ldots f_q)^{lN}\cdot f^{c_1}_1\ldots f^{c_q}_q$, where $c_i\geq 0$, $\sum_i c_i\cdot w_{f_i} = w$, $\sum_i c_i=l$, then ${\Delta}_f$ contains a ball $B_{w,l}$ of dimension $\dim L(X,\mathcal E)$ and radius $lN\cdot s$, which depends on $(w,l)\in {\Lambda}_p$ but not on particular numbers $c_i$. Hence
$$
\#\left( M\cap \bigcap_{\substack{f=(f_1\ldots f_q)^{lN}\cdot f^{c_1}_1\ldots f^{c_q}_q \\ c_i\geq 0, \;\sum_i c_i\cdot w_{f_i} = w, \;\sum_i c_i=l}} {\Delta}_f \right) \geq C\cdot l^{\dim L(X,\mathcal E)},
$$
where $C > 0$ does not depend on $(w,l)\in {\Lambda}_p$.\\

This implies that
$$
\limsup\limits_{l\rightarrow \infty} \frac{\dim H^{0}(X,\operatorname{Sym}^{p(Nql+l)}{\mathcal E})}{(Nql+l)^{d}} \geq
$$
$$
\geq \limsup\limits_{l\rightarrow \infty} \left( \frac{1}{((Nq+1)l)^{d}} \cdot \dim \operatorname*{Span}_{\substack{c_1+\ldots +c_q = l \\ c_i\geq 0}} \{ {\chi}^{-u}\otimes f \mid u \in {\Delta}_f\cap M,\; f=(f_1\ldots f_q)^{lN}\cdot f^{c_1}_1\ldots f^{c_q}_q \} \right)\geq
$$

$$
\geq \frac{C}{(Nq+1)^d} \cdot {\alpha}(p,X,\mathcal E) > 0,
$$
and so $\mathcal E$ is big by definition.

\end{proof}

\begin{corollary}
Suppose $\mathcal E$ is a toric vector bundle on a projective toric variety $X=X(\Delta)$. If there is an integer $a>0$ and $f\in \operatorname{Sym}^{a}{E}\setminus \{ 0 \}$ such that $\dim {\Delta}_f = \dim X$, then $\mathcal E$ is big.
\end{corollary}

\begin{proof}

If $b>0$ is large enough, then $M\cap {\Delta}_{f^b\cdot g} \neq \varnothing$ for any $g\in {\epsilon} (X,\operatorname{Sym}^{a}{\mathcal E})$. As the subset 
$$
\bigcup_{p>0} {\epsilon}(X,\operatorname{Sym}^{p}{\mathcal E}) \subset \operatorname{Sym}{E}
$$
contains a vector space basis of the symmetric algebra and $\dim L(X,\mathcal E) = \dim X$, ${\alpha}(a+ab,X,\mathcal E) > 0$.

\end{proof}

\begin{example}
Suppose $\mathcal E = \mathcal O (D_1) \oplus \cdots \oplus \mathcal O (D_t)$, $t>0$, is a \textit{split} toric vector bundle on a projective toric variety $X=X(\Delta)$. Then $\mathcal E$ is big if and only if $\dim L(X,\mathcal E) = \dim X$ if and only if there are nonnegative integers $a_1,\ldots , a_t$ such that $\mathcal O (a_1D_1+\ldots + a_tD_t)$ is big.
\end{example}

Indeed, let $\bigcup_{p>0} {\epsilon}(X,\operatorname{Sym}^{p}{\mathcal E}) \subset \operatorname{Sym}{E}$ be the set of all nonconstant monomials in $x_1,\ldots ,x_t$, where $x_1,\ldots , x_t\in E$ are generators of the splitting summands of $\mathcal E$. For any $p>0$, ${\mathcal A}(p,X,\mathcal E)$ is a subalgebra of the symmetric algebra $\operatorname{Sym}{E}$:
$$
\bigoplus_{w\in \Lambda , (w,l)\in {\Lambda}_p} {\mathcal A}_{(w,l)}(p,X,\mathcal E)\subset \operatorname{Sym}^{pl}{E},\;\; l>0.
$$
Hence ${\alpha}(p,X,\mathcal E) = 0$ unless $\dim L(X,\mathcal E) = \dim X$.\\

\section*{Acknowledgement}

The author is grateful to the School of Mathematics of Sun Yat-sen University for support and excellent working conditions.\\

\bibliographystyle{ams-plain}

\bibliography{BigToricVB}

\end{document}